\documentclass[12pt]{amsart}

\usepackage[english]{babel}
\usepackage[T1]{fontenc}
\usepackage[latin1]{inputenc}
\usepackage{lmodern}
\usepackage{a4}
\usepackage[T1]{fontenc}
\usepackage[all]{xy}
\usepackage{verbatim}
\usepackage{stmaryrd}
\usepackage{amsfonts}
\usepackage{amsmath,amssymb,amsthm}
\usepackage{enumerate}
\usepackage{xspace}
\usepackage{ulem}
\usepackage{euscript}
\usepackage{epsfig}
\usepackage{fancyhdr}
\usepackage{txfonts}
\usepackage{babel,indentfirst}
\usepackage{graphicx}

\makeatletter \@addtoreset{equation}{section}

\theoremstyle{definition}

\newtheorem{thm}{Theorem}[section]

\newtheorem{lem}{Lemma}[section]
\newtheorem{pro}{Proposition}[section]

\newtheorem{cor}{Corollary}[section]

\newtheorem{rem}{Remark}[section]

\newcommand{\B}{\mathcal{B}}

\newcommand{\C}{\mathbb{C}}

\title{Jordan-product  commuting nonlinear maps  with  $\lambda$-Aluthge transform }

\author{F. Chabbabi and M. Mbekhta}

\address{Universit\'e  Lille1, UFR de Math\'ematiques, Laboratoire CNRS-UMR 8524 P. Painlev\'e, 59655 Villeneuve Cedex, France}
\email{fadilino@gmail.com}
\email{Mostafa.Mbekhta@math.univ-lille1.fr}

\keywords{normal,  quasi-normal operators, polar decomposition, $\lambda$-Aluthge transform.\\
This work was supported in part by the Labex CEMPI (ANR-11-LABX-0007-01).}

\date{}

\begin{document}
\maketitle

\begin{abstract}
Let $H$ and $K$ be two complex Hilbert spaces and $\B(H)$ be the algebra of  bounded linear operators from $H$ into itself. 

The main purpose in this paper is to obtain a characterization of bijective  maps $\Phi  :\B(H) \to \B(K)$   satisfying the following condition    
\begin{equation*}
\Delta_{\lambda} (\Phi(A)\circ \Phi(B))=\Phi(\Delta_{\lambda}(A\circ B))  \: \text{ for all} \; A, B\in \B(H),
 \end{equation*}
 where   $\Delta_{\lambda}(T)$ stands the $\lambda$-Aluthge transform of the operator $T\in\B(H)$ and $A\circ B=\frac{1}{2}(AB+BA)$ is the Jordan product of  $A$ and $B$. We prove that a bijective map $\Phi$ satisfies the above condition, if and only if there exists an unitary operator $U:H\to K$, such that $\Phi$ has the form   $\Phi(A)=UAU^*$ for all $A\in\B(H)$.
\end{abstract}

\section{Introduction }

 Let $H$ and $K$ be  two complex Hilbert spaces  and  $\mathcal{B}(H, K)$ be the Banach space of all bounded linear operators from $H$ into $K$. In the case
  $K=H$, $\mathcal{B}(H,H)$  is simply denoted by  $\B(H) $ and is a  Banach algebra.
 
  For an arbitrary operator $T\in \B(H, K)$, we denote by $\mathcal{R}(T)$, $\mathcal{N}(T)$ and  $T^*$  the range,  the null subspace and   the operator adjoint  of $T$ respectively. 
  For $T\in \B(H)$, the spectrum of $T$ is denoted by $\sigma(T)$.
 
 An operator $T\in \mathcal{B}(H,K)$ is  a {\it partial isometry} when $T^*T$ is an orthogonal projection (or, equivalently $TT^*T = T$). 
  In particular $T$ is an {\it isometry} if $T^*T=I_{H}$, and {\it unitary} if $T$ is a surjective isometry.

As usually, for $T\in \B(H)$ we denote the module of $T$ by $|T|=(T^*T)^{1/2}$ and we shall always write, without further
mention, $T=V|T|$ to be the unique polar decomposition  of $T$, where $V$ is the appropriate partial isometry satisfying  $\mathcal{N}(V)=\mathcal{N}(T)$. 
The Aluthge transform was   introduced in \cite{alu}  as 
 $$\Delta(T)=|T|^{\frac{1}{2}}V|T|^{\frac{1}{2}}, \quad  T\in \B(H),$$
   to  extend some properties  of hyponormal  operators.
Later, in \cite{oku},    Okubo introduced a more general notion called  $\lambda-$Aluthge
transform which has also been studied in detail.

 For $\lambda \in [ 0, 1]$, the $\lambda-$Aluthge transform is  defined by, 
 $$
 \Delta_{\lambda}(T)=|T|^{\lambda}V|T|^{1-\lambda}, \quad T\in \B(H).
 $$
  Notice  that $\Delta_0(T) =V|T|=T$, and $\Delta_1(T) = |T|V$ which is known as Duggal's
  transform.  It has since been studied in many different contexts and considered by a
  number of authors (see for instance,  \cite{ alu, ay, ams1, ams2,  kp1, kp2,  kp3} and some
  of the references there).  One of the interests  of the Aluthge transform lies in the fact that
   it respects many   properties of the original operator. For example, 
\begin{equation}
   \sigma_*(\Delta_{\lambda}(T)) = \sigma_*(T),   \mbox{ for every } \; \; T \in \B(H), 
\end{equation} 
 where $  \sigma_*$ runs over a large family of spectra. See \cite[Theorems 1.3, 1.5]{kp3}.
 Another important property is that $Lat(T)$, the lattice of $T$-invariant subspaces of
 $H$, is nontrivial if and only if $Lat(\Delta(T))$ is nontrivial
(see \cite[Theorem 1.15]{kp3}). 

   In \cite{bmg},   the authors described  the   linear  bijective mappings on von Neumann algebras which commute with the $\lambda$-Aluthge transform
      $$\Delta_{\lambda}(\Phi (T))=\Phi (\Delta_{\lambda}(T )) \mbox{ for   every }  T \in \B(H). $$

 In \cite{chf} the first author gives a complete description of the bijective maps  $\Phi :\B(H)\to\B(K)$ which satisfies  following condition,
  \begin{equation}
\Delta_{\lambda}(\Phi (A)\Phi (B))=\Phi (\Delta_{\lambda}(A B))   \mbox{ for   every }   A, B\in \B(H),
 \end{equation}
for some  $\lambda \in]0,1[$.

 \bigskip
 
  In this paper we will be concerned with the  Jordan product  commuting maps  with the  $\lambda$-Aluthge transform  in the following sense,
  \begin{equation}\label{c1}
\Delta_{\lambda}(\Phi(A)\circ \Phi(B))=\Phi(\Delta_{\lambda}(A \circ B)) \: \mbox{ for  all}\;  A, B\in \B(H),
 \end{equation}
 where $A\circ B=\frac{1}{2}(AB+BA)$ is the Jordan product of  $A$ and $B$.

Our main result gives a complete description of the bijective maps $\Phi :\B(H)\to\B(K)$ which satisfies Condition (\ref{c1}), the precise theorem being as follows

  \begin{thm} \label{th1}
 Let $H$ and $K$ be  two  complex Hilbert space, such that $H$ is of dimension greater than $2$. Let   $\Phi :\B(H)\to\B(K)$ be a bijective map. Then
  $\Phi$ satisfies  (\ref{c1}) for some  $\lambda \in]0,1[\:$   if and only if  there exists  a unitary operator $U : H\to K $ such that  
 $$\Phi(A)=UAU^*\;\;\;\;\text{ for every} \; A\in\B(H).$$
 \end{thm}

 The paper is organized as follows. In the second section, we establish some useful results on the Aluthge transform. 
  These results are needed for proving our main theorem in the  section 3.

\section{  Some properties of the  Aluthge transform}

In this section we establish some results,  properties of the  Aluthge transform. These results are necessary for the proof of the main theorem.

We first recall some basic notions that are used in the sequel. An operator $T\in\B(H)$  is  normal if $T^*T=TT^*$,  and  is  quasi-normal, if  it commutes with $T^*T$ ( i.e. $TT^*T=T^*T^2$), or equivalently  $|T|$ and $V$ commutes ($T = V\vert T\vert$  a  polar decomposition of $T$). In finite dimensional spaces  every quasi-normal operator  is normal.  It is easy to see that if $T$ is quasi-normal, then $T^2$ is also quasi-normal, but the converse is false as shown by nonzero  nilpotent operators .

Also, it is well known that  the quasi-normal operators  are exactly the fixed points of  $\Delta_{\lambda}$ (see \cite[Proposition 1.10]{kp3}). That is,
   \begin{equation}\label{qn}
 T \; \; \mbox{ quasi-normal} \; \;  \iff  \Delta_{\lambda}(T) = T.
 \end{equation}

  For nonzero $x,y \in H$ , we denote by  $x\otimes y$ the rank one operator defined by 
 $$(x\otimes y)u=<u,y>x \: \mbox{   for } \:  u\in H.$$
  It is easy to show that every rank one   operator has  the previous form and that   $x\otimes y$ is an orthogonal projection, if and only if $x=y$ and $\|x\|=1$. 

\bigskip

 We start with  the following proposition.  It is found in \cite{chf}.  To facilitate the reading we include  the proof.

  \begin{pro}\label{p1} (\cite{chf})  Let $x,y \in H$ be nonzero vectors. We have 
 $$\Delta_{\lambda}(x\otimes y)=\dfrac{<x,y>}{\|y\|^2}(y\otimes y)\: \:  \mbox{ for every }  \: \lambda \in ]0,1[.$$
\end{pro}
\begin{proof}
Denote $T=x\otimes y$, then  
$$T^*T=|T|^2=\|x\|^2(y\otimes y)=\big(\frac{\|x\|}{\|y\|}(y\otimes y)\big)^2\;\;\text{and}\;\;|T| = \frac{\|x\|}{\|y\|}(y\otimes y).$$
It follows that   
$$|T|^2=\|x\|\, \|y\| |T|\;\; \text{and}\;\; |T|^\gamma =(\|x\|\, \|y\|)^{\gamma-1}|T| \: \text{ for every}  \; \gamma>0.$$

Now,  let  $T=V|T|$ be the polar decomposition of $T$. We have 
\begin{align*}
\Delta_{\lambda}(x\otimes y)=\Delta_{\lambda}(T)&=|T|^{\lambda}V|T|^{1-\lambda}\\
&=(\|x\|\, \|y\|)^{{\lambda-1}}(\|x\|\, \|y\|)^{{-\lambda}}|T|V|T|\\
&=\dfrac{1}{\|x\|\, \|y\|}|T|T =\frac{1}{\|y\|^2}(y\otimes y)\circ (x\otimes y)\\
&=\dfrac{<x,y>}{\|y\|^2}(y\otimes y).
\end{align*}
Thus , $\Delta_{\lambda}(x\otimes y)=\dfrac{<x,y>}{\|y\|^2}(y\otimes y)\: $ as desired
\end{proof}

 \begin{pro}\label{l1} Let $A\in \B(H)$ and  $P=x\otimes x$ be a rank one projection on $H$. Then
$$\Delta_{\lambda}(A \circ P) = P  \;\;\text{if and only if}\;\; PA=P.$$
\end{pro}
   \begin{proof} The "if" part  follows directly from the previous proposition. We show the "only if" part.  Suppose that $\Delta_{\lambda}(A \circ P) = P$.
   First, it is easy to see the following 
       $$ PA = P\:  \iff \: A^*x = x.$$
       So,  to complete the proof of the proposition, it suffices to show that $A^*x = x$.

  Put  $T = A \circ P $. Then  $T = \frac{1}{2}(Ax\otimes x+x\otimes A^*x)$,  and thus
 \begin{equation}\label{eq-}
T^2=\frac{1}{4}(<Ax,x>Ax\otimes x + Ax\otimes A^*x+<A^2x,x> x\otimes x+<Ax,x> x\otimes A^*x).
 \end{equation}
  In the other hand, we have
 $$
  \sigma(T)=\sigma(\Delta_{\lambda}(T))=\sigma(P)=\{0,1\}  \; \; \mbox{ and thus} \;  \sigma(T^2)=\{0,1\}.
   $$
  Since the rank of $T$ is at most $2$,  $tr(T^2)=tr(T)=1$.  We also have $tr(T)=<Ax,x>$ and $tr(T^2)=\frac{1}{2}(<A^2x,x>+<Ax,x>)$. Therefore
  $$<Ax,x> =  <A^2x,x> =1.$$
  It follows that 
  $Tx=\frac{1}{2}(Ax+x)$ and $T^*x=\frac{1}{2}(A^*x+x)$. Hence 
  $$<Tx,x> = 1 \: \mbox{and} \: <T^*x,x>=1.$$
  
  From (\ref{eq-}) we get  
\begin{equation} \label{eq1}
 T^2=\frac{1}{4}(Ax+x)\otimes (A^*x+x)=Tx\otimes T^*x.
\end{equation}

  Let  $T=U|T|$ be the polar decomposition of $T$.  It holds
  \begin{eqnarray*}
  T^2 &=& U|T|U|T|= U|T|^{1-\lambda}(|T|^{\lambda}U|T|^{1-\lambda})|T|^{\lambda}\\
  &=& U|T|^{1-\lambda} \Delta_{\lambda}(T)|T|^{\lambda} =  U|T|^{1-\lambda}(x \otimes  x) |T|^{\lambda}   \\
  & = &    U|T|^{1-\lambda}x \otimes |T|^{\lambda} x.
 \end{eqnarray*}
   From  (\ref{eq1}) we get 
   \begin{equation}\label{eq+}
Tx\otimes T^*x = U|T|^{1-\lambda}x \otimes |T|^{\lambda}x.
   \end{equation}
 Since $<T^*x,x> =1$, we get that  
  $$Tx = <x,|T|^{\lambda}x> U|T|^{1-\lambda}x.$$
 Thus
 $$ |T|^{\lambda}|T|x = |T|^{\lambda}U^*Tx =  <x,|T|^{\lambda}x>U^* U|T|x = <x,|T|^{\lambda}x>|T|x.$$ 
  By function calculus, we obtain 
   $$|T||T|x=( <x,|T|^{\lambda}x>)^{1/\lambda}|T|x.$$
    It follows that
  $$<|T|^2x,x>=\||T|x\|^2=( <x,|T|^{\lambda}x>)^{1/\lambda}<|T|x,x>.$$ 
  
   Now, by Holder-Mc Carthy inequality (see \cite{fur} page 123), we deduce
 
 $$\big{\|}|T|x\big{\|}^2=( <x,|T|^{\lambda}x>)^{1/\lambda}<|T|x,x>\leq (<|T|x,x>)^2\leq \big{\|}|T|x\big{\|}^2.$$
  Hence $\big{\|}|T|x\big{\|}=<|T|x,x>$. Therefore  there exists $\alpha \in \C$ such that $|T|x=\alpha x$.\\
  The last equality and Equation  (\ref{eq+}) imply
  $$Tx\otimes T^*x=\alpha Ux \otimes x.$$
Thus, $T^*x = \beta x$ for some $\beta \in \C$.  Since,  $<T^*x, x> = 1, \;$ it holds $ \; \beta = 1$.  On the other hand $T^*x=\frac{1}{2}(A^*x+x)$. Then
$ \frac{1}{2}(A^*x+x) = x$. Finally we conclude that  $A^*x=x$ and hence  $PA=P$. The proof is completed.
   \end{proof}
   
  \bigskip 
   
   The following lemma gives a property  of  rank one projections
 
    \begin{pro}\label{l2} Let $A\in \B(H)$ and  $P=x\otimes x$ be a rank one projection on $H$. Then
 $$\Delta_{\lambda}(A\circ P)=A \;\text{ if and only if} \;\; A=\alpha P \;\;\text{for some}\; \alpha \in \C.$$
 \end{pro}

 \begin{proof}
   The "if" part  is obvious. We show the "only if" part. Put $T= A\circ P$. Then 
   $$T =\frac{1}{2}(Ax\otimes x+x\otimes A^*x) \; \;  \mbox{and} \; \;   T^* =\frac{1}{2}(x\otimes Ax + A^*x\otimes x).$$
   By assumption,  $A = \Delta_{\lambda}(T)$.  
  Since  $\mathcal{R}( \Delta_{\lambda}(T)), \mathcal{R}( \Delta_{\lambda}(T)^*) \subseteq \mathcal{R}(T^*)$,  we have
  
 $$\mathcal{R}(A), \mathcal{R}(A^*)\subseteq \mathcal{R}(T^*)\subseteq H_0 := span\{x,A^*x\}.$$  
 Note that  $H_0$ is an invariant subspace of $A$ and $A^*$. 
  
  {\bf  Claim:}   there exists $\delta\in\mathbb{C}$  such that $A^*x=\delta x$. 
  In this case, $A$ and $A^*$ are rank one operators. Moreover, their ranges  $\mathcal{R}(A)$ and  $\mathcal{R}(A^*)$ are generated by $x$. It is easy to deduce from this, that
   $A=\overline{\delta}P.$
  
  We prove the claim by contradiction. 
Assume on contrary that $A^*x$ and $x$ are linearly independent. Let $\{x, e\}$ be an orthonormal basis of $H_0$.  We can choose $e \in H_0$  such that  $<e,Ax>\geq 0$. 
In this basis $A$ has the following form
 \begin{equation}\label{eq3}
A=x\otimes A^*x +e\otimes A^*e.  
 \end{equation}
 Let us consider the following cases : 
 
 \textbf{Case 1 :}  $A^*e=0$. In this case,  we show $A^*x$ and $x$ are linearly dependent, which is a contradiction. 

From (\ref{eq3}), we have $A=x\otimes A^*x$. Hence $$T=P\circ A= 1/2 ( x\otimes A^*x+<Ax,x>x\otimes x)=1/2 x\otimes\big(A^*x+<A^*x,x>x\big).$$
Since $\mathcal{R}(A) \subseteq \mathcal{R}(T^*)$, then  $x$ and $A^*x+<A^*x,x>x$ 
are linearly dependent it follows that $x$ and $A^*x$ are also linearly dependent,  which is a contradiction. 
 
\textbf{Case 2 :}  $A^*e\ne 0$. In this case,  we show that $A=2T$ which is a contradiction  with the fact that  
$\Delta_{\lambda}(T)=\Delta_{\lambda}(A/2)=A,$ and 
$$\|A\|=\|\Delta_{\lambda}(A/2)\|\leq \frac{\|A\|}{2}.$$ 
Thus  $A=0$,  which is not possible.
  
  Note that since $A=\Delta_{\lambda}(T)$ it holds 
   $$\|A\|=\|\Delta_{\lambda}(T)\|\leq  \|T\|  =   \|\frac{1}{2}(Ax\otimes x+x\otimes A^*x)\|\leq\frac{1}{2}(\|Ax\|+\|A^*x\|)\leq \|A\|.$$
  Hence
   \begin{equation}\label{eq2}
  \|A\|=\|Ax\|=\|A^*x\|.
  \end{equation}  

Now, by (\ref{eq3}), we have 

$$
A=2T \Leftrightarrow e\otimes A^*e=Ax\otimes x \Leftrightarrow
\left\{
      \begin{aligned}
        Ax= a e ;\\
        A^*e= \bar{a} x,\\       
      \end{aligned}
    \right. \;\;  \mbox { for some } a \in \C.
  $$
So, we just need to prove  the equivalence.
First,  we show that 
$$A^*e=\|A^*e\| x.$$

 Indeed by assumption, we get 
$$tr(A)=tr(\Delta_{\lambda}(T)) = tr(T) = <Ax,x>.$$
  By (\ref{eq3}), we get $$tr(A)=<Ax,x> + <Ae,e>.$$  
  Thus $<Ae,e>=0$.  Since $A^*e \in H_0$ and $<e,Ax>\geq 0$, it follows that
 \begin{equation}\label{eq4}
A^*e=<A^*e,x>x=\|A^*e\|x. 
 \end{equation}
 
Second, we show $$Ax=\|A^*e\| e.$$ Indeed 
 From (\ref{eq3}), we get 
 \begin{equation*}
  AA^*=x\otimes AA^*x +e\otimes AA^*e.
 \end{equation*}
 Thus
$$AA^*x=\|A^*x\|^2 x+<x,AA^*e>e,$$
moreover, by (\ref{eq2})  it follows that 
  $$\|AA^*x\|^2=\|A^*x\|^4 +|<x,AA^*e>|^2\leq \|AA^*\|^2=\|A\|^4=\|A^*x\|^4.$$
 Therefore
 $$<x,AA^*e>=0.$$
 This, together  with the fact $A^*e=\|A^*e\| x$,  show that 
   $$ AA^*x=\|A\|^2 x \;\;\mbox{ and}\;\;  AA^*e=\|A^*e\|^2e=\|A^*e\| Ax.$$
 It follows that 
\begin{equation} \label{eq4'}
 Ax=\|A^*e\|e.
\end{equation}

   This completes the proof.
  \end{proof}

 \begin{rem}  If $T = V \vert T\vert \in \B(H)$ is the polar decomposition of $T$,  then $V$ is unitary if and only if  $T$ and $T^*$ are one-to-one.
  
 \end{rem}  
  
   \begin{lem}\label{l33} Let $T\in\B(H)$ and $\lambda \in]0,1[$. Suppose that $T$ and $T^*$ are one-to-one. Then,
  $$\Delta_{\lambda}(T^2)=T \;\;  \Longrightarrow  \;\; T^2 =  T^*.$$ 
  \end{lem}
  \begin{proof}  Let us consider $T^2=U|T^2|$ the polar decomposition of $T^2$. Since $T$ and $T^*$ are  injective, then $T^2$ and $(T^2)^*$ are also injective.
   Thus $U$ is unitary operator.  We have
  \begin{align*}
   \Delta_{\lambda}(T^2) =  |T^2|^{\lambda}U|T^2|^{1-\lambda} = T & \Longrightarrow   |T^2|^{\lambda}U|T^2|^{1-\lambda}   |T^2|^{\lambda} = T  |T^2|^{\lambda}\\
    & \Longrightarrow   |T^2|^{\lambda}U |T^2| = T  |T^2|^{\lambda}\\
     & \Longrightarrow   |T^2|^{\lambda}T^2 = T  |T^2|^{\lambda}.
  \end{align*}
  On the other hand, $T \Delta_{\lambda}(T^2)=T^2$, hence
    \begin{align*}  T |T^2|^{\lambda}U|T^2|^{1-\lambda} = T^2 = U|T^2| = U  |T^2|^{\lambda}|T^2|^{1-\lambda} 
   \end{align*}
   Since  $T^2$ is injective, we get
   $$ T |T^2|^{\lambda}U = U  |T^2|^{\lambda}.$$
   Thus, 
   $$  |T^2|^{\lambda} T^2 U = U  |T^2|^{\lambda} \; \; \text{and} \; \;   |T^2|^{\lambda}  T^2= U  |T^2|^{\lambda} U^* \geq 0.$$
 It follows that
   \begin{align*}
    |T^2|^{\lambda}  T^2 = T^{*2}  |T^2|^{\lambda} & =  |T^2| U^*  |T^2|^{\lambda} \\
     & =  |T^2|^{\lambda}   |T^2|^{1 - \lambda} U^*  |T^2|^{\lambda} \\
     & =  |T^2|^{\lambda}  \Delta_{\lambda}(T^2) ^* \\
     & =  |T^2|^{\lambda} T^{*}.
  \end{align*}
  Finally we obtain  $ T^2 = T^*$, as claimed.
    \end{proof}
   
     \begin{lem}\label{l22} Let $S\in\B(H)$ and $\lambda \in]0,1[$. Suppose that $S$ and $S^*$ are one-to-one. Then,
  $$\Delta_{\lambda}(S)=S^* \;\;  \Longrightarrow  \;\; S =  S^*.$$ 
  \end{lem}
  \begin{proof}  Let us consider $S=U|S|$ the polar decomposition of $S$. Since $S$ and $S^*$ are  injective,  $U$ is unitary operator.  We have,
 $$\Delta_{\lambda}(S) \Delta_{\lambda}(S)^* = S^*S = |S|^2 =  |S|^{\lambda}   |S|^{1 - \lambda}  |S|^{1 - \lambda}   |S|^{\lambda}.$$
 By a simple calculation, we deduce
 $$ U  |S|^{2(1 -\lambda)}  = |S|^{2(1 -\lambda)}U$$
 By the continuous functional calculus, we obtain $S$ quasi-normal. \\
 Consequently, $ S = \Delta_{\lambda}(S) = S^*$, as desired.
  \end{proof}

 \begin{rem}  
  It is well known that  $T$ is quasi-normal  does not imply that his adjoint $T^*$  is always quasi-normal. 
  
  For example :  take the right Shift operator $S$ on a separable Hilbert  space $H$,  with orthonormal bas is  $(e_n)_n$.  We have 
  $$S^*S=I \: \text{ and} \: SS^*=I-P_1$$
   where  $P_1$ is the projection onto  the span of the first vector $e_1$.  Hence 
   $$S^*SS^*=S^* \: \text{and} \:  S(S^*)^2=S^*-P_1S^*\neq S^*.$$
       This shows that $S^*$ is not quasi-normal.
       
       \bigskip
       
        Now, we replace the condition $S$ and $S^*$ are injective in preceding lemma by $S$ is quasi-normal. We have the following result
  
   \end{rem}

 \begin{lem}\label{l22}
  Let $S\in\B(H)$ and $\lambda \in]0,1[$. Suppose that $S$ is quasi-normal. Then
  $$\Delta_{\lambda}(S^*)=S \;\;  \Longrightarrow  \;\; S =  S^*.$$ 
 \end{lem}
  \begin{proof}
  Let us consider     the polar decomposition of $S$ and $S^*$,       $S=U|S|$ and $S^*=U^*|S^*|$  respectively. 
It is  well known that  $|S^*|^p=U|S|^pU^*$ for all $p > 0$. 

Now, since $S$ is quasi-normal,  $U|S|^p=|S|^pU$ and  $U^*|S|^p=|S|^pU^*$ for every  $p>0$. Hence,   we have the following equalities 
 \begin{eqnarray*}
S =  \Delta_{\lambda}(S^*)&=&|S^*|^{\lambda}U^*|S^*|^{1-\lambda}\\&=&U|S|^{\lambda}U^*U^*U|S|^{1-\lambda}U^*\\&=& U(U^*)^2|S|. 
 \end{eqnarray*} 
It follows that 
 \begin{equation}
 U(U^*)^2|S|=U|S|.
 \end{equation}
 After multiplying this equality by $U^*$, we get
  $$(U^*)^2|S|=|S|.$$ 
  Hence 
 $$(U^*)^2|S|=|S|U^2=U|S|U=|S|.$$ Thus 
 $$S= U|S| = |S|U=U^*U|S|U=U^*|S|=|S|U^*=S^*.$$
 Thus $ S = S^*$ and the lemma is proved.
 \end{proof}
  
  Combining  Lemmas \ref{l22} and \ref{l33}, we obtain the following corollary which plays an important role in the proof of $\Phi(I) = I$.
  
   \begin{cor}\label{l3}
  Let $T\in\B(H)$ and $\lambda \in]0,1[$. Suppose that $T$ and $T^*$ are one-to-one. Then,
  $$\Delta_{\lambda}(T^2)=T  \;\;\text{if and only if }\;\; T=I.$$ 
  \end{cor}
  
  \bigskip

  The next lemma,  was established by S. Garcia in the case $\lambda = \frac{1}{2}$, see \cite{gar}. It  identifies the kernel of the $\lambda$-Aluthge transform as the set of all operators which are nilpotent of order two. We need this lemma later.

  \begin{lem}\label{00} 
Let $\lambda \in ]0,1]$ and $T\in \B(H)$. Then 
$$ {\Delta_{ \lambda}}(T)=0 \: \text{ if and only if} \: T^2 = 0.$$
\end{lem}

\begin{proof}
 Let $T=V|T|$ be the polar decomposition of $T$. 

If $\;  {\Delta_{ \lambda}}(T) = 0$, then we have
$$T^2=V|T|V|T|=V|T|^{1-\lambda} {\Delta_{ \lambda}}(T)|T|^{\lambda}=0,  \: \text{and thus} \;   T^2 = 0.$$

Conversely, suppose that  $T^2=0$. Then  
$V|T|V|T|= T^2 = 0$.  Hence $V^*V|T|V|T|=0$, which  implies $|T|V|T|=0$,  since $V^*V$ is the projection onto $\overline{{\mathcal{R}(|T|)}}$. 
Whence $|T|V$ vanishes on $\overline{{\mathcal{R}(|T|)}}$. Furthermore $\mathcal{N}(V) = \mathcal{N}(|T|)$, from which one concludes $|T|V$ vanishes on $\mathcal{N}(|T|)$. 
Consequently, $|T|V=0$. Thus  ${\Delta_{ \lambda}}(T) = |T|^{\lambda}V|T|^{1-\lambda} = 0$.
\end{proof}

  \section{ Proof of the main theorem}

   An  idempotent self adjoint  operator $P\in\B(H)$  is said to be an orthogonal projection.  
 Clearly  quasi-normal idempotents  are orthogonal projections. \\
Two  projections $P,Q\in\B(H)$ are said to be orthogonal if 
$$PQ=QP=0,$$
 in this case, we denote $P\perp Q$. 
  A partial ordering between orthogonal projections is defined as follows,  
  $$Q\leq P\:  \;  \mbox{  if  } \:  \; PQ = QP = Q.$$

\begin{rem}  We have
$$ P\perp Q \: \iff \: P+Q \; \text{orthogonal projection},$$
and
   $$ Q \leq P \:  \iff PQ + QP = 2Q \: \iff P\circ Q = Q .$$
   \end{rem}

 \bigskip
  
\begin{pro}\label{cor1} Let $\Phi:\B(H)\to\B(K)$ be a bijective map satisfying  (\ref{c1}). Then
 $$\Phi(0)=0 \quad \text{ and} \quad  \Phi(I)=I.$$
\end{pro}
 \begin{proof}
  Since $\Phi$ is onto, let $A\in \B(H)$ such that $\Phi(A)=0$. By  (\ref{c1}) we have  
 $$\Phi(0)=\Delta_{\lambda}(\Phi(0)\circ \Phi(A)) = \Delta_{\lambda}(0) = 0.$$
 This proves $\Phi(0) = 0$.
 
 New, we prove that $\Phi(I) = I$. For the sake of simplicity,   write  $T=\Phi(I)$. First, we show that $T$ and $T^*$ are injective. Indeed, let $y\in K$ such that $Ty=0$. Since $\Phi$ is onto, there exists $B\in\B(H)$ such that $\Phi(B)=y\otimes y$. By  (\ref{c1}) we get 
  $$\Delta_{\lambda}((y\otimes y)\circ T) =  \Delta_{\lambda}(\Phi(B)\circ T)=\Delta_{\lambda}(\Phi(B)\circ \Phi(I))=\Phi(\Delta_{\lambda}(B)).$$
 This can be rewritten in the other therms,
  $$\frac{1}{2}\Delta_{\lambda}(Ty\otimes y+y\otimes T^*y)= \frac{1}{2}\Delta_{\lambda}(y\otimes T^*y)=\Phi(\Delta_{\lambda}(B)).$$ 
 In the other hand  $(y\otimes T^*y)^2=<Ty,y>y\otimes T^*y=0$, since $Ty = 0$. By Lemma \ref{00},  $\Delta_{\lambda} (y\otimes T^*y)=0$ and thus $\Phi(\Delta_{\lambda}(B))=0$. Therefore $\Delta_{\lambda}(B)=0$, because $\Phi$ is bijective  and $\Phi(0) = 0$. Again, by  Lemma \ref{00}, $B^2=0$. 
 
 Now, using  Condition (\ref{c1}),  we get
  $$\|y\|^2 y\otimes y=\Phi(B)^2=\Delta_{\lambda}(\Phi(B)\circ \Phi(B))=\Phi(\Delta_{\lambda}(B^2))=0,$$
   which implies that $y=0$ and whence $T$ is injective. With a similar argument we prove that $T^*$ is  injective. 
 
Again from (\ref{c1}), we get 
$$\Delta_{\lambda}(T^2)=\Delta_{\lambda}(\Phi(I)\circ \Phi(I))=\Phi(\Delta_{\lambda}(I))=\Phi(I)=T.$$
From Corollary \ref{l3}, we deduce that $T=\Phi(I)=I$, and completes the proof.
\end{proof}

 \begin{thm}\label{cor2} Let $\Phi:\B(H)\to\B(K)$ be a bijective map satisfying  (\ref{c1}). Then
\begin{enumerate}[(i)]
\item$\Delta_{\lambda}(\Phi(A))=\Phi(\Delta_{\lambda}(A))$, for all $A\in\B(H)$. In particular $\Phi$ preserves the set of quasi-normal operators in  both directions. 
\item $\Phi (A^2)=(\Phi (A))^2$ for all  $A$ quasi-normal.
\item $\Phi $ preserves the set of  orthogonal projections.
\item $\Phi $ preserves the orthogonality between the projections ; $$P\perp Q \Leftrightarrow \Phi (P)\perp\Phi (Q).$$
\item $\Phi $ preserves the order relation on the set of orthogonal projections in  both directions ; 
 $$Q\leq P \Leftrightarrow \Phi (Q) \leq \Phi  (P).$$
\item  $\Phi (P+Q)=\Phi (P)+\Phi (Q)$ for all orthogonal projections $P,Q$ such that $P\perp Q$. 
\item  $\Phi $ preserves the set of  rank one  orthogonal projections in  both directions.
\end{enumerate}
 \end{thm} 
 \begin{proof}
 (i)  Since, from Proposition \ref{cor1}, $\Phi(I) = I$,  for  $B=I$ in (\ref{c1})  we obtain  $\Delta_{\lambda}(\Phi(A))=\Phi(\Delta_{\lambda}(A))$, for all $A\in\B(H)$. 
  Now, by  $(2.1)$ we get that  
$\Phi$ preserves the set of quasi-normal operators in  both directions. 
 
 (ii)  Let $A\in \B(H)$ be a quasi-normal operator. Since $\Phi$ preserves the set of quasi-normal operators,   
  $\Phi(A), \Phi(A^2), (\Phi(A))^2$  are quasi-normal. By (\ref{c1}), we get  $\Delta_{\lambda}((\Phi(A))^2)=\Phi(\Delta_{\lambda} (A^2))$. 
  Hence  $(\Phi(A))^2=\Phi(A^2)$ follows from $(2.1)$.

 (iii) It is an immediate consequence of (i) and (ii).
 
 \bigskip

 In the rest of the proof $P, Q$ are two orthogonal  projections.
 
 (iv) Assume that $P,Q$ are orthogonal ($P\perp Q)$. Since $\Phi$ preserves the set of orthogonal  projections,  $\Phi(P)\circ \Phi(Q)$ is self-adjoint operator. 
 Hence by (\ref{c1}) and $(2.1)$, we get  
 $$\Phi(P)\circ \Phi(Q)=\Delta_{\lambda}(\Phi(P)\circ\Phi(Q))=\Phi(\Delta_{\lambda}(P\circ Q))=0.$$
 Thus  
 \begin{equation}\label{eq7}
 \Phi(P)\Phi(Q)+\Phi(Q)\Phi(P)=0
 \end{equation}
 Multiplying  (\ref{eq7}) by $\Phi(P)$ on  both sides,  we get 
 \begin{equation}\label{eq8}
 \Phi(P)\Phi(Q)+\Phi(P)\Phi(Q)\Phi(P)=0 \;\;\text{ and }\;\;
\Phi(P)\Phi(Q)\Phi(P)+\Phi(Q)\Phi(P)=0.
  \end{equation}
 It follows that $\: \Phi(P)\Phi(Q)=\Phi(Q)\Phi(P)=0$.
 
 (v) Suppose that $Q \leq P$.  Using Remark 3.1 and   (\ref{c1}), we get 

$$\Phi(P)\circ \Phi(Q) = \Delta_{\lambda}(\Phi(P)\circ \Phi(Q)) = \Phi(\Delta_{\lambda}(P\circ Q)) = \Phi(Q).$$
Hence,
 $$\Phi(Q)\leq \Phi(P).$$
  (vi) Since  $P\perp Q$,  $P+Q$ is an orthogonal  projection and $P,Q\leq P+Q$. By (v),   $\Phi(P),\Phi(Q) \leq \Phi(P+Q)$. Thus $\Phi(P)+\Phi(Q)\leq \Phi(P+Q)$. 
  Since $\Phi^{-1}$ satisfies the same assumptions as $\Phi$, it follows that $\Phi(P)+\Phi(Q)=\Phi(P+Q)$.
 
 (vii)  Let $P=x\otimes x$ be a rank one projection. Then $\Phi(P)$ is a non zero projection. Let $y\in K$ be an unit vector such that 
 $y\otimes y \leq \Phi(P)$.  Thus $\Phi^{-1}(y\otimes y) \leq P$.  Since $P$ is a minimal projection and $\Phi^{-1}(y\otimes y)$ is a non zero projection, 
  $\Phi^{-1}(y\otimes y)= P$. Therefore  $\Phi(P)=y\otimes y$ is a rank one projection. 
 \end{proof}
 
\bigskip 

 In the following proposition we introduce a function  $h:\C\to\C$, which will be used later to prove the linearity of $\Phi$.

 \begin{pro}\label{cor3}   Let $\Phi:\B(H)\to\B(K)$ be a bijective map satisfying  (\ref{c1}). Then there exists a bijective function $h:\C\to\C$ such that 
$$h(0)=0\;\; \text{and}\;\;h(1)=1, $$
 and it  satisfies the following properties : 
 \begin{enumerate}[(i)]
 \item $\Phi(\alpha I)=h(\alpha)I$ for all $\alpha \in \C$.
 \item $h(\alpha\beta)=h(\alpha)h(\beta)$ for all $\alpha, \beta \in \C$.
 \item $h(-\alpha)=-h(\alpha)$ for all $\alpha \in \C$.
 \item $\Phi(\alpha A)=h(\alpha)\Phi(A)$ for all $A$ quasi-normal and $\alpha \in \C$.
 \end{enumerate}
 \end{pro}   
  \begin{proof}
  First, we show that there exists a function 
  $$h : \C \to  \C \: \text{such that}  \: \Phi(\alpha I)=h(\alpha)I\: \text{ for all} \: \alpha \in\C.$$
  
 When  $\alpha\in\{0,1\}$ the function $h$ is defined by $h(\alpha)=\alpha$ since $\Phi(\alpha I)=\alpha I$.
 
  Now, we suppose that $\alpha \ne 0$. Let  $y\in K$ an arbitrary unit vector.  By Theorem 3.1 (iii), there exists $P = x\otimes x \in \B(H)$ with $\Vert x\Vert = 1$,  such that  $\Phi(P)=y \otimes y$. \\
   By (\ref{c1}) with $A=P$ and $B=\alpha P$ we get 
  $$\Delta_{\lambda}\Big( \Phi(\alpha P) \circ \Phi(P)\Big) =\Phi(\Delta_{\lambda}(\alpha P))=\Phi(\alpha P).$$
 From  Proposition \ref{l2}, there exists $h_P(\alpha) \in \C$ such that 
 $$\Phi(\alpha P)=h_P (\alpha) P.$$
 Since $\Phi$ is bijective and $\alpha \not=0$,   $h_P(\alpha)\not=0$.
 
 Again using  (\ref{c1}) with $A=\alpha I$ and $B=P$, we get 
 $$\Delta_{\lambda}\Big(\Phi(\alpha I) \circ \Phi(P)\Big)=\Phi(\alpha P)=h_p(\alpha) \Phi(P).$$
 From Proposition \ref{l1},  we get that
  \begin{eqnarray*}
 \Phi(P)\Phi(\alpha I) = h_P(\alpha)\Phi(P) = y\otimes \Phi(\alpha I)^*y=h_P(\alpha)y\otimes y
  \end{eqnarray*}
 Then  
  $\Phi(\alpha I)^* y=\overline{h_P(\alpha)}y$ for every $y\in K$.
  It follows that $\Phi(\alpha I)=h_P(\alpha)I$. Thus $h := h_P$ does not depend on $P$. 
  
  On the other hand, since $\Phi$ is bijective and $\Phi^{-1}$ satisfies the same properties as $\Phi$,  we conclude that  the function $h:\C\to\C$ is well defined,  bijective   and   $\Phi(\alpha I)=h(\alpha)I$. 
 
(ii)  Let $\alpha, \beta \in \C$,  taking  $A=\alpha I$ and $B=\beta I$ in (\ref{c1}),  we get
 $$h(\alpha)h(\beta)I=\Delta_{\lambda}(\Phi(\alpha I)\circ \Phi(\beta I))=\Phi(\Delta_{\lambda}(\alpha I\circ \beta I)=h(\alpha\beta) I.$$
 Therefore $h$ satisfies the point (ii). 
 
 From (ii),  we have $(h(-1))^2=1$.  Hence $h(-1)=-1$,  since $h(1)=1$ and $h$ is bijective. Hence  
 $$h(-\alpha)=h(-1)h(\alpha)=-h(\alpha) \: \text{ for all} \: \alpha\in\C.$$  
 Thus (iii) is satisfied.

 To prove (iv), let us consider $A$ quasi-normal and $\alpha \in \C$. By (\ref{c1}), we have $$h(\alpha)\Phi(A)=\Delta_{\lambda}(h(\alpha)\Phi(A))=\Delta_{\lambda}(\Phi(A)\circ \Phi(\alpha I))=\Phi(\Delta_{\lambda}(\alpha A))=\Phi(\alpha A).$$
 Thus  (iv) holds. 
 \end{proof}

  \begin{lem}\label{l4}
 Let $P=x\otimes x $, $P'=x'\otimes x'$ be two rank one orthogonal projections such that $P\perp P'$ (that is $<x,x'> = 0$). Then 
$$\Phi(\alpha P+\beta P')=h(\alpha)\Phi(P)+h(\beta)\Phi(P'),  \: \text{for every}  \:   \alpha ,\beta \in \C$$
 \end{lem}
 \begin{proof} If $\alpha =0$ or $\beta=0$ the result follows from preceding Proposition. 
 
 Suppose that $\alpha\ne 0$ and $\beta \ne 0$.   Note that $(\alpha P+\beta P')\circ (P+P')=\alpha P+\beta P'$. \\
 By  (\ref{c1}),  applied to  $B=\alpha P+\beta P'$ and $A=P+P'$, we obtain
 \begin{eqnarray*}
\Phi(\alpha P+\beta P')&=&\Phi(\Delta_{\lambda}(\alpha P+\beta P'))\\&=&\Phi(\Delta_{\lambda}((\alpha P+\beta P')\circ(P+P')))\\&=&\Delta_{\lambda}(\Phi(\alpha P+\beta P')\circ\Phi(P+P'))\\&=&\Delta_{\lambda}(\Phi(\alpha P+\beta P')\circ(\Phi(P)+\Phi(P'))).
\end{eqnarray*}  
It follows that 
\begin{equation}\label{eq9'}
\Phi(\alpha P+\beta P')=\Delta_{\lambda}\Big(\Phi(\alpha P+\beta P')\circ\Phi(P)+\Phi(\alpha P+\beta P')\circ\Phi(P')\Big)
\end{equation}
Again by  (\ref{c1}) with $B=\alpha P+\beta P'$ and $ A=P$,  we get that
\begin{eqnarray*}
\Delta_{\lambda}\Big(\Phi(\alpha P+\beta P')\circ\Phi(P)\Big)&=&\Delta_{\lambda}(\Phi((\alpha P+\beta P')\circ P)))\\
&=&\Phi(\Delta_{\lambda}(\alpha P)) \quad \quad \text{(since $(\alpha P+\beta P')\circ P=\alpha P$)}\\
&=&\Phi(\alpha P)=h(\alpha)\Phi(P). 
  \end{eqnarray*} 
  By Proposition \ref{l1} and  $h(\alpha)\ne 0$, it follows that 
  \begin{equation}\label{eq10}
 \Phi(P)\Phi(\alpha P+\beta P')=h(\alpha)\Phi(P).
  \end{equation}
  In the same manner, we also have 
   \begin{equation}\label{eq11}
 \Phi(P')\Phi(\alpha P+\beta P')=h(\beta)\Phi(P').
  \end{equation}
  Now, we denote by $T=\Phi(\alpha P+\beta P')$, $\Phi(P)=y\otimes y$ and $\Phi(P')=y'\otimes y'$. \\
  Since $P$ and $P'$ are orthogonal, then $y$ and $y'$ are orthogonal too.\\
    By (\ref{eq10}) and (\ref{eq11}), we get that 
  $$T^*y=\overline{h(\alpha)} y \quad \text{and} \quad T^*y'=\overline{h(\beta)}y',$$
   and also
  \begin{equation}\label{eq12}
\Phi(\alpha P+\beta P')\circ\Phi(P)=\frac{1}{2}(h(\alpha)\Phi(P)+\Phi(\alpha P+\beta P')\Phi(P))=\frac{1}{2}\big((h(\alpha)y+Ty)\otimes y\big), 
  \end{equation}
  and
   \begin{equation}\label{eq13}
   \Phi(\alpha P+\beta P')\circ\Phi(P')=\frac{1}{2}(h(\beta)\Phi(P')+\Phi(\alpha P+\beta P')\Phi(P'))=\frac{1}{2}\big((h(\alpha)y'+Ty')\otimes y'\big).
     \end{equation}
 By (\ref{eq13}), it follows that 
  \begin{equation}\label{eq14}
  T=\frac{1}{2}\Delta_{\lambda}\Big(\big(h(\alpha)y+Ty\big)\otimes y+\big(h(\alpha)y'+Ty'\big)\otimes y'\Big). 
  \end{equation} 
 This implies that  the rank of $T$ and $T^*$ are less than two. Moreover $\mathcal{R}(T), \mathcal{R}(T^*)\subseteq span\{y,y'\}$.  Since $T^*y=\overline{h(\alpha )}y$ and $T^*y'=\overline{h(\beta )}y'$, then $T= h(\alpha )y\otimes y+h(\beta )y'\otimes y'$. Therefore 
 $$T=\Phi(\alpha P+\beta P')=h(\alpha)\Phi(P)+h(\beta)\Phi(P'),$$
 which is the desired equality.
 \end{proof} 
 
 \bigskip

 If  $\Phi$ satisfies  (\ref{c1}), then $\Phi$ preserves the set of rank one projection and  preserves also  the orthogonality between projections (see Theorem \ref{cor2} (iii) and (iv)). 
 Then  for two vectors $x,x' \in H$ such that $\|x\|=\|x'\|=1$ and $<x,x'>=0$, there exist two vectors $y,y' \in K$ such that 
 $$\| y\|=\|y'\|=1, \: <y, y'>=0, \: \: \text{and} \:  \: \Phi(x\otimes x)=y\otimes y, \: \Phi(x'\otimes x')=y'\otimes y'.$$
  With the preceding notations, we have the following lemma 
  
 \begin{lem}\label{l5}  Let $\Phi:\B(H)\to\B(K)$ be a bijective map satisfying  (\ref{c1}).
Then there exists $\mu \in\C$ such that $|\mu|=1$ and  
$$\Phi(x\otimes x'+x'\otimes x)=\mu y\otimes y'+ \bar{\mu}y'\otimes y).$$
In addition, we obtain
$$h(2)=2 \;\;  \text{and }\;\;h(\frac{1}{2}) = \frac{1}{2}.$$
\end{lem}

\begin{proof} Let us denote by $A=x\otimes x'+x'\otimes x$. First, note that  $A$ is self adjoint operator  of rank  two, and we have
$$A^2=x\otimes x+x'\otimes x',$$
which is an non-trivial orthogonal projection. Therefore, the spectrum of $A$ is $\sigma(A)=\{-1,0,1\}$. Hence we can find $f_1, f_2$ two unit and orthogonal vectors from 
$H$ ~~($\|f_1\|=\|f_2\|=1$ and $<f_1,f_2>=0$) \: such that 
\begin{equation}\label{eq3.9}
A= f_1 \otimes f_1-f_2 \otimes f_2.
\end{equation}

 Let us denote  $T=\Phi(A)$. By Lemma \ref{l4}, we have 
$$T=\Phi(f_1 \otimes f_1-f_2 \otimes f_2)=\Phi(f_1 \otimes f_1)+h(-1)\Phi(f_2 \otimes f_2)=\Phi(f_1 \otimes f_1)-\Phi(f_2 \otimes f_2),$$
 Now, since $\Phi(f_1 \otimes f_1)$ and $\Phi(f_2 \otimes f_2)$ are two orthogonal projections, then we deduce that  $T$ is a self adjoint operator of rank two, and  
$$tr(T)=tr(\Phi(f_1 \otimes f_1))-tr(\Phi(f_2 \otimes f_2))=0.$$

In the other hand, since $A$  is self adjoint, by   Theorem \ref{cor1} (ii) and (vi) , we get
 $$T^2= \Phi(A)^2=\Phi(A^2)=\Phi(x\otimes x+x'\otimes x') = y\otimes y + y'\otimes y'.$$
Moreover, 
\begin{equation}\label{eq15}
T^2=y\otimes y+y'\otimes y'.
\end{equation}
It follows that $y,y'\in \mathcal{R}(T)$, and thus 
$$\mathcal{R}(T)=span\{y,y'\}.$$

Clearly,  $A=2 A \circ (x\otimes x)$.  By (\ref{c1}), we get  
\begin{eqnarray*}
T=\Delta_{\lambda}(T)&=&\Delta_{\lambda}(\Phi(2 A \circ (x\otimes x))\\&=&h(2)\Phi(\Delta_{\lambda}(A\circ (x\otimes x)))\\&=&h(2)\Delta_{\lambda}(\Phi(A)\circ \Phi(x\otimes x)))\\&=&h(2)\Delta_{\lambda}(T\circ (y\otimes y)).
\end{eqnarray*}
Therefore 
\begin{equation}\label{eq16}
T=\frac{h(2)}{2}\Delta_{\lambda}(Ty\otimes y+y\otimes T^*y)=\frac{h(2)}{2}(Ty\otimes y+y\otimes Ty).
\end{equation}
If follows that 
$$tr(T)=h(2)<Ty,y>=0.$$
 Hence $<Ty,y>=0$. Since $\{y,y'\}$ is an orthonormal basis of $\mathcal{R}(T)$ then there exists $\mu \in\C$ such that $Ty=\mu y'$.
  From (\ref{eq16}) we get 
 \begin{equation}
 T=\frac{h(2)}{2}(\mu y'\otimes y+ \bar{\mu}y\otimes y'). 
 \end{equation}
 
  In particular, we also have $Ty=\frac{h(2)}{2} \mu y'=\mu y'$, it follows that 
  $$h(2)=2 \quad \text{and thus} \quad h(\frac{1}{2}) = \frac{1}{2} \; \:  \text{(since $h$ multiplicative and $h(1) = 1$)} .$$
Hence 
$$T=\mu y'\otimes y+ \bar{\mu }y\otimes y',$$
and we have 
$$T^2=|\mu|^2 (y\otimes y+y'\otimes y')=y\otimes y+y'\otimes y'.$$
 Hence $|\mu|=1$. This completes the proof.
\end{proof}

As a direct consequence from the preceding lemma, we have the following corollary.
 \begin{cor}
The function $h:\C\to \C$ defined in Proposition \ref{cor3} is additive.
\end{cor}
\begin{proof}
Let $x,x'\in H$ be two unit and orthogonal vectors for $H$ and $\alpha, \beta \in \C$. 
Denote 
$$P=x\otimes x,  \; P'=x'\otimes x', \: \:  Q=\Phi(P)=y\otimes y, \; Q'=\Phi(P')=y'\otimes y'.$$
Put  $\; A=x\otimes x'+x'\otimes x$,   and $B=\alpha P + \beta P' $.
Observe that 
$$A\circ P=\frac{1}{2} A\; \; \text{and}\;\; A\circ P'=\frac{1}{2}A,$$ 
thus  
$$A\circ B=\dfrac{\alpha+\beta}{2} A.$$
 By Lemma \ref{l5},  there exists $\mu \in\C$ such that $|\mu|=1$ and  
 $$\Phi(A)=\mu y\otimes y'+ \bar{\mu} y'\otimes y.$$
 Then, by a simple  calculation, we get
 $$ \Phi(P)\circ\Phi( A)= Q\circ\Phi( A)=\frac{1}{2}\Phi( A).$$ 
  Hence, using (\ref{c1}), we immediately obtain the following
\begin{eqnarray*}
h(\frac{\alpha+\beta}{2})\Phi(A)&=& \Phi(\frac{\alpha+\beta}{2}A) = \Phi(\Delta_{\lambda}(\dfrac{\alpha+\beta}{2}(A)))
\\&=& \Phi(\Delta_{\lambda}( B\circ A)) = \Delta_{\lambda}(\Phi(B)\circ\Phi( A))\\&=&
\Delta_{\lambda}\Big(\Phi(\alpha P + \beta P')\circ\Phi( A)\Big)\\&=&\Delta_{\lambda}\Big(\Phi(\alpha P)\circ\Phi( A) + \Phi(\beta P')\circ\Phi( A)\Big)\\&=&
\Delta_{\lambda}\Big(h(\alpha) \Phi(P)\circ\Phi( A) + h(\beta) \Phi(P')\circ\Phi( A)\Big)
 \\&=&\frac{1}{2} ( h(\alpha)+h(\beta))\Phi(A).
\end{eqnarray*}
Since $h(1/2)=1/2$, we have
$$\frac{1}{2} h(\alpha +\beta)=h(\dfrac{\alpha+\beta}{2})=\frac{1}{2}(h(\alpha)+h(\beta)).$$
Thus $h(\alpha+\beta)=h(\alpha)+h(\beta)$. It follows  that $h$ is additive.
\end{proof}

\bigskip

Now, we are in position to prove our main Theorem. \\

{\it   \bf Proof of Theorem \ref{th1}.} 
\begin{proof} The "if" part is  immediate.  

We show the "only if" part. Assume that $\Phi :\B(H)\to\B(K)$ is bijective and  satisfies  (\ref{c1}).   The proof of theorem is organized in several steps.

\bigskip

\begin{enumerate}[\bf Step. 1]
\item For every  $A\in\B(H)$, for all $x,y \, \text{such that} \; \Vert x\Vert = \Vert y\Vert = 1 \, \text{and} \,  \Phi (x\otimes x)=y\otimes y$, we have 
 \begin{equation}\label{eq}
<\Phi (A)y,y>=h(<Ax,x>), 
  \end{equation}

  Put  $\; T=  2(A\circ (x\otimes x)) = Ax\otimes x+x\otimes A^*x$. 
  Let  $\alpha_1, \alpha_2$ be the non zero  eigenvalues  of $T$. By the Schur decomposition  of $T$, there exist two unit and orthogonal vectors $e_1,e_2$ such that
 $$T=\alpha_1 e_1\otimes e_1+\alpha_2 e_2\otimes e_2+ \beta e_1\otimes e_2.$$ 
First, we show that 
  \begin{equation}\label{eq'}
tr(\Delta_{\lambda}(\Phi(T)))=2h(<Ax,x>). 
  \end{equation}
 If $\alpha_1=\alpha_2=0$, we have $T=\beta e_1\otimes e_2$,  thus $T^2 = 0$ and $\Delta_{\lambda}(T) = 0$, by Lemma 2.4. 
 Hence $\Delta_{\lambda}(\Phi(T))=\Phi(\Delta_{\lambda}(T))=0$, since $\Phi$ commutes with $\Delta_{\lambda}$. Consequently,
 $$tr(T)=2<Ax,x> = 0 = tr(\Delta_{\lambda}(\Phi(T))),$$
  and, in this case,   (\ref{eq'}) is satisfied. 
 
 Now suppose that $(\alpha_1,\alpha_2)\not=(0,0)$.
 
  First, note that    
 $$T\circ (e_1\otimes e_1)=\frac{1}{2}(e_1\otimes (\bar{\beta} e_2+2\bar{\alpha_1} e_1))=\frac{1}{2}(e_1\otimes v), $$
 
 with  $v=\bar{\beta} e_2+2\bar{\alpha_1} e_1 \;\;\text{and }\;\; <e_1,v>=2\alpha_1.$
From Proposition 2.1, it follows that  
$$\Delta_{\lambda}(T\circ (e_1\otimes e_1))=\frac{1}{2}\Delta_{\lambda}((e_1\otimes v)=\alpha_1\frac{1}{\|v\|^2}(v\otimes v) .$$
By  (\ref{c1}), we get 
\begin{eqnarray*}
h(\alpha_1)\Phi(\frac{1}{\|v\|^2}(v\otimes v))&=& \Phi(\Delta_{\lambda}(T\circ (e_1\otimes e_1)))\\&=&\Delta_{\lambda} (\Phi(T)\circ \Phi(e_1\otimes e_1)).
\end{eqnarray*} 
 We have also 
$$T\circ ( e_2\otimes e_2)=\frac{1}{2}(\beta e_1+2\alpha_2 e_2)\otimes e_2. $$ 
Hence  
$$\Delta_{\lambda}(T\circ ( e_2\otimes e_2))=\alpha_2 e_2\otimes e_2.$$
Again (\ref{c1}) implies that  
\begin{eqnarray*}
h(\alpha_2)\Phi(e_2\otimes e_2)&=&\Phi(\Delta_{\lambda}(T\circ ( e_2\otimes e_2)))\\&=&\Delta_{\lambda} (\Phi(T)\circ \Phi(e_2\otimes e_2))
\end{eqnarray*}
Now, observe  that $T\circ (e_1\otimes e_1 +e_2\otimes e_2)=T$.  We apply again (\ref{c1}), then 
\begin{eqnarray*}
\Delta_{\lambda}(\Phi(T))&=&\Phi(\Delta_{\lambda}(T\circ (e_1\otimes e_1 +e_2\otimes e_2)))\\&=&
\Delta_{\lambda}(\Phi(T)\circ \Phi(e_1\otimes e_1 +e_2\otimes e_2))\\&=&
\Delta_{\lambda}(\Phi(T)\circ \Phi(e_1\otimes e_1) +\Phi(T)\circ\Phi(e_2\otimes e_2)). 
\end{eqnarray*}
Since $\Phi(\frac{1}{\|v\|^2}(v\otimes v))$ and $\Phi(e_2\otimes e_2)$ are orthogonal projections,   then  
\begin{eqnarray*}
tr(\Delta_{\lambda}(\Phi(T)))&=&tr(\Delta_{\lambda}(\Phi(T)\circ \Phi(e_1\otimes e_1) +\Phi(T)\circ\Phi(e_2\otimes e_2)))\\&=&tr(\Phi(T)\circ \Phi(e_1\otimes e_1) +\Phi(T)\circ\Phi(e_2\otimes e_2))\\&=&tr(\Phi(T)\circ \Phi(e_1\otimes e_1))+tr( \Phi(T)\circ\Phi(e_2\otimes e_2))\\&=&tr(h(\alpha_1)\Phi(\frac{1}{\|v\|^2}(v\otimes v))+tr(h(\alpha_2)\Phi(e_2\otimes e_2))\\&=&h(\alpha_1)+h(\alpha_2)\\&=&h(\alpha_1+\alpha_2)=h(tr(T))\\&=&2h(<Ax,x>).
\end{eqnarray*}
Now, we show that 
\begin{equation}\label{eq''}
tr(\Delta_{\lambda}(\Phi(T))=2<\Phi(A)y,y>. 
\end{equation}
By (\ref{c1}) and  $h(2)=2$,  we get 
\begin{eqnarray*}
\Delta_{\lambda}(\Phi(T))&=&\Delta_{\lambda}(\Phi(2A\circ (x\otimes x))=\Delta_{\lambda}(h(2)\Phi(A\circ (x\otimes x))\\&=&h(2)\Phi(\Delta_{\lambda}(A\circ (x\otimes x)))\\&=&2\Delta_{\lambda}( \Phi(A)\circ \Phi(x\otimes x))\\&=&2\Delta_{\lambda}( \Phi(A)\circ (y\otimes y)).
\end{eqnarray*}
Thus
\begin{eqnarray*}
 tr(\Delta_{\lambda}(\Phi(T)) &=& 2tr(\Delta_{\lambda}( \Phi(A)\circ (y\otimes y)))=2tr( \Phi(A)\circ (y\otimes y))\\
 &=& 2<\Phi(A)y,y>.
\end{eqnarray*}
From (\ref{eq'}) and (\ref{eq''}), it  follows  that
$$ <\Phi(A)y,y>=h(<Ax,x>),$$
 for all $A\in\B(H)$ and for all $x\in H$ and $y\in H$  such $\Phi(x\otimes x)=y\otimes y$, which completes the proof of (3.13).

\bigskip

 \item The function $h$ is continuous. 

  Let   $\mathcal{E}$ be a bounded subset  in $\C$ and  $A\in\B(H)$  such that  $\mathcal{E}\subset W(A)$, where $W(A)$ is the numerical range of A.
By   (\ref{eq}),  
 \begin{equation*}
h(\mathcal{E})\subset h(W(A))   = W(\Phi (A))
 \end{equation*}
  Since  $W(\Phi (A))$ is bounded,    $h$ is bounded on the bounded subset, which implies that $h$ is continuous,  since it  is additive (see Corollary 3.1). 
   We then have, using Proposition 3.2 (ii) and Corollary 3.1,  that $h$ is an automorphism continuous   over the complex field $\C$.
    It follows that $h$ is the identity or the  complex conjugation map   (see, for example, \cite{ks}). 

\bigskip

\item $\Phi $ is linear or anti-linear. 
  
 Let  $y\in K$ and $x\in H$ be unital  vectors such that $y \otimes y=\Phi ( x\otimes x)$. Let $\alpha\in \C$ and $A ,B\in \B(H)$ be arbitrary. Using (\ref{eq}), we get   
  \begin{eqnarray*}
 <\Phi (A+B)y,y>&=&h(<(A+B)x,x>)\\&=&h(<Ax,x>+<Bx,x>)\\&=&h(<Ax,x>)+h(<Bx,x>)
 \\&=&<\Phi (A)y,y>+<\Phi (B)y,y>\\&=&<(\Phi (A)+\Phi (B))y,y>,
  \end{eqnarray*}
and 
\begin{eqnarray*}
 <\Phi (\alpha A)y,y>&=&h(<\alpha A x,x>)\\
 &=& h(\alpha)h(<Ax,x>) \\
 &=& h(\alpha)<\Phi (A)y,y>.
\end{eqnarray*}
   
  Hence, we immediately obtain the following 
   $$\forall A,B\in\B(H), \forall  \alpha \in \C, \; \;   \Phi (A+B)=\Phi (A)+\Phi (B)  \: \text{and} \: \Phi (\alpha A)=h(\alpha)\Phi (A).$$ 
   Therefore  $\Phi $ is linear or anti-linear,  since $h$ is the identity or the complex conjugation. 
   
   \bigskip
   
\item $\Phi (A)=UAU^*$  every $A\in \B(H)$, for some  unitary  operator  $U\in \B(H,K)$. 
   By  Step.3, we have  $\Phi$ or  $\Phi^*$ is linear, where $\Phi^*$  is defined by $\Phi^*(A)=(\Phi(A))^*$.  From Theorem 3.1 (i),  $\Phi$  commute
    with $\Delta_{\lambda}$. Now, by \cite[Theorem 1]{bmg},   there exists a unitary operator $V: H\to K$, such that  $\Phi$ take one of the following forms 
   \begin{equation}\label{f1}
   \Phi(A)=VAV^* \;\; \text{for all}\;\; A\in\B(H),
   \end{equation}
 either 
   \begin{equation}\label{f2}
   \Phi(A)=VA^*V^* \;\; \text{for all}\;\; A\in\B(H).
   \end{equation}
   In order to complete the proof we have to show $\Phi $ can not take the form (\ref{f2}).
   
   Seeking a contradiction,   suppose that (\ref{f2}) holds. Then  for every $\; A\in\B(H)$,
   \begin{eqnarray*}
   {\Delta_{\lambda }}(A^{*})&=&  {\Delta_{\lambda }}(V^*\Phi(A)V) \\
   &=& V^*{\Delta_{\lambda }}(\Phi(A)) V\\
   &=&  V^*\Phi({\Delta_{\lambda }}(A))V\\
   &=& ({\Delta_{\lambda }}(A))^*.
   \end{eqnarray*}
   
Therefore 
\begin{equation}\label{ee}
{\Delta_{\lambda }}(A^{*})= ({\Delta_{\lambda}}(A))^{*},\quad \mbox{ for every} \; A\in\B(H).
\end{equation}
Now, let us consider  $A=x\otimes x'$ with $x, x'$  are  unit,  independent and non-orthogonal vectors  in $H$. Then $A^*=x'\otimes x$ and  by Proposition \ref{p1}, we have 
 $$ ({\Delta_{\lambda}}(A))^*=<x', x>(x'\otimes x') \; \;  \mbox{and} \; \; {\Delta_{\lambda}}(A^{*})=<x', x>(x\otimes x),$$
which contradicts  (\ref{ee}).  So if we take $U = V$,  $\; \Phi$  gets the  form 
$$ \Phi(A)=UAU^*,  \;\; \text{for all}\;\; A\in\B(H).$$
\end{enumerate}
 This completes the proof of our main theorem.
\end{proof}

\section{Star Jordan Product Commuting  Maps  }

We end this paper by  characterizing  the  bijective maps   $\Phi :\B(H)\to\B(K)$,  which verify   
  \begin{equation}\label{c2}
\Delta_{\lambda}(\Phi(A)\circ (\Phi(B))^*)=\Phi(\Delta_{\lambda}(A \circ B^*)) \quad \text{for all} \; \;  A, B\in \B(H).
 \end{equation}

  \begin{thm} \label{th3}
 Let $H$ and $K$ two  complex Hilbert space, such that $H$ is of dimension greater than $2$. Let   $\Phi :\B(H)\to\B(K)$ be a bijective map. 
 Then $\Phi$ satisfies  (\ref{c2}), if and only if, there exists  an unitary operator $U : H\to K $ such that  
 $$\Phi(A)=UAU^*\;\;\;\;\text{ for all $A\in\B(H)$.}$$
 \end{thm}

�\begin{rem}  Based on the arguments and the methods  developed in the proof of Theorem 1.1,  it is not difficult to proof Theorem 4.1. For example, 
 If $\Phi :\B(H)\to\B(K)$ is a bijective map, satisfying (\ref{c2}), then 
 \begin{enumerate}[(i)]
 \item $\Phi(0)=0.$,
 \item $\Phi(I)=I.$,
 \item $\Delta_{\lambda}(\Phi(A))=\Phi(\Delta_{\lambda}(A))$ and $\Delta_{\lambda}((\Phi(A))^*)=\Phi(\Delta_{\lambda}(A^*))$ for all $A\in\B(H)$,
 \item $\Phi$ preserve the set of quasi-normal operator in both directions,
 \item $\Phi$ preserves the set of self adjoint operators in both directions. 
 \end{enumerate}

\begin{proof}
 (i) Since $\Phi$ is onto, then there exists  $B\in \B(H)$ such that $\Phi(B)=0$. Hence, by (\ref{c2}) $\Phi(0)=\Delta_{\lambda}(\Phi(0)\circ \Phi(B)^*)=0$. 
 
 (ii) Let us denote $T:=\Phi(I)$. First we show that $T$ is one-to-one. Let $y\in K$ such that $Ty=0$. Now since $\Phi$ is onto, there exists  $B\in\B(H)$ such that 
 $\Phi(B)=y\otimes y$. By (\ref{c2}), we have $\Delta_{\lambda}(T\circ y\otimes y)=y\otimes y$.  Thus $y\otimes y=0$ and  it follows that $y=0$. 
 
 Again from  (\ref{c2}) it  follows that $T\circ T^*=T$. Hence $T$ is self adjoint and $T^2=T\circ T^*=T$. Thus $T=I$ since $T$ is one-to-one.
 
 (iii) is immediate, and (iv) is deduced directly from (iii).
 
 (v) Let $B$ be a self adjoint operator, and $S:=\Phi(B)$.  By (iii) $S$ is quasi-normal operator, and by (\ref{c2}) with $A=I$ we get, that $\Delta(S^*)=S$. 
 By Lemma \ref{l22}, $\; S=S^*$ and thus $\Phi$ preserves the self adjoint operators.
  \end{proof}

The rest of the proof of Theorem 4.1 is similar to that of Theorem 1.1. We do not give details of those arguments.

\end{rem}

\end{document}